\documentclass[a4paper,11pt,reqno]{amsart}

\usepackage{xcolor}

\usepackage[margin=2.6cm]{geometry}
\usepackage{braket}

\usepackage{enumerate}
\usepackage{amssymb, amsmath, bm, mathabx}
\usepackage{mathrsfs}
\usepackage{amscd}
\usepackage[active]{srcltx}
\usepackage{verbatim}
\PassOptionsToPackage{hyphens}{url}
\usepackage[colorlinks,linkcolor={black},citecolor={black},urlcolor={black}]{hyperref}

\definecolor{jan}{rgb}{0.0,0.3,0.8}
\definecolor{mat}{rgb}{0.0,0.5,0.3}

\usepackage[latin1]{inputenc}

\theoremstyle{plain}
\newtheorem{theorem}{Theorem}[section]
\newtheorem{corollary}[theorem]{Corollary}
\newtheorem{lemma}[theorem]{Lemma}
\newtheorem{proposition}[theorem]{Proposition}

\newtheorem{definition}[theorem]{Definition}
\newtheorem{assumption}[theorem]{Assumption}

\newtheorem*{definition*}{Definition}

\theoremstyle{remark}
\newtheorem{remark}[theorem]{Remark}
\newtheorem{example}[theorem]{Example}

\newtheorem*{claim*}{Claim}
\newtheorem*{remark*}{Remark}
\newtheorem*{example*}{Example}
\newtheorem*{notation*}{Notation}

\numberwithin{equation}{section}

\def\N{{\mathbb N}}

\def\R{{\mathbb R}}
\def\C{{\mathbb C}}

\def\M{{\mathbb M}}

 \newcommand{\Dens}{{\mathfrak P}}

\newcommand{\one}{{{\bf 1}}}

\newcommand{\sfN}{{\mathsf{N}}}

\newcommand{\eps}{\varepsilon}
\renewcommand{\phi}{\varphi}

\newcommand{\dd}{\; \mathrm{d}}

\DeclareMathOperator{\Span}{span}

\DeclareMathOperator{\Tr}{Tr}

\newcommand{\ip}[1]{\langle {#1}\rangle}

\DeclareMathOperator{\Ent}{Ent}

\newcommand{\rms}{\mathrm{s}}
\newcommand{\rmS}{\mathrm{S}}

\def\tand{\quad{\rm and}\quad}

\newcommand{\ddt}{\frac{\mathrm{d}}{\mathrm{d}t}}

\newcommand{\cM}{\mathscr{M}}

\newcommand{\cL}{\mathscr{L}}

\newcommand{\cK}{\mathscr{K}}

\newcommand{\cX}{\mathcal{X}}

\newcommand{\cA}{\mathcal{A}}

\newcommand{\cP}{\mathscr{P}}

\renewcommand{\tilde}{\widetilde}

\newcommand{\cno}[1]
{\gamma_1 \cdots \widecheck{\gamma_{#1}} \cdots \gamma_N }
\newcommand{\cnominus}[1]
{\gamma_1 \cdots \widecheck{\gamma_{#1}} \cdots \gamma_{N-1} }
\newcommand{\cnoo}[2]
{\gamma_1 \cdots \widecheck{\gamma_{#1}} \cdots \widecheck{\gamma_{#2}} \cdots \gamma_N }

\begin{document}


\title
[Gradient flows]
{Characterisation of gradient flows for a given functional}
\address{
Institute of Science and Technology Austria (ISTA)\\
Am Campus 1\\ 
3400 \newline Klos\-ter\-neu\-burg\\ 
Austria}
\author{Morris Brooks}
\email{morris.brooks@ist.ac.at}

\author{Jan Maas}
\address{
Institute of Science and Technology Austria (ISTA)\\
Am Campus 1\\ 
3400 \newline Klos\-ter\-neu\-burg\\ 
Austria}
\email{jan.maas@ist.ac.at}

\begin{abstract}
Let $X$ be a vector field and $Y$ be a co-vector field on a smooth manifold $M$. 
Does there exist a smooth Riemannian metric 
        $g_{\alpha \beta}$
    on $M$ such that $Y_\beta = g_{\alpha \beta} X^\alpha$?
The main result of this note gives necessary and sufficient conditions for this to be true.
As an application of this result we show that a finite-dimensional ergodic Lindblad equation admits a gradient flow structure for the von Neumann relative entropy if and only if the condition of \textsc{bkm}-detailed balance holds.
\end{abstract} 

\maketitle

\section{Introduction}
\label{sec:intro}

This paper deals with the following general question:
\begin{quote}
    \emph{Let $X^\alpha \in \Gamma(T M)$ be a vector field and 
    $Y_\beta \in \Gamma(T^* M)$ be a co-vector field on 
    a smooth manifold $M$.    
    Does there exist a smooth Riemannian metric 
        $g_{\alpha \beta}$
    on $M$ such that $Y_\beta = g_{\alpha \beta} X^\alpha$?} 
    \footnote{Throughout the paper we use index notation and Einstein's summation convention. 
    Greek letters denote abstract indices, Roman letters denotes concrete indices.}
\end{quote}
Clearly, this is not always true: $X^\alpha$ and $Y_\beta$ will have to satisfy some compatibility conditions. 
Firstly, $X^\alpha$ and $Y_\beta$ need to have the same set of zeroes (critical points).
Secondly, at all other points $m \in M$, they need to satisfy $X^\alpha Y_\alpha|_m > 0$.
A third (and slightly less obvious) compatibility condition is obtained by differentiating the equation 
    $Y_\beta = g_{\alpha \beta} X^\alpha$:
at each critical point $m \in M$ 
there should exist a scalar product 
        $\bar g_{\alpha \beta} 
        \in T_m^* M \otimes_{\rm S} T_m^* M$
such that
    $
        \nabla_\alpha Y_\gamma|_m
        =
        \bar g_{\beta \gamma} 
            \nabla_\alpha X^\beta|_m
    $
for some (equivalently, any) connection 
    $\nabla_\alpha$.
This condition does not hold automatically: 
it represents a compatibility constraint on $X^\alpha$ and $Y_\beta$ 
with a natural interpretation in some examples below.

While these three conditions are clearly necessary, 
it is not obvious that they are also sufficient. 
The main result of this paper shows that this is indeed the case,
under mild smoothness and non-degeneracy assumptions;
namely,
at all critical points, 
we require 
non-degeneracy of the derivative of $Y_\beta$ 
and we assume that
$X^\alpha$ and $Y_\beta$ are real analytic in suitable local coordinates;
cf.~Section \ref{sec:main-result} for the details.

\begin{theorem}[Main result]
    \label{thm:main-intro}
    Let 
        $X^\alpha \in \Gamma(T M)$ 
    and 
        $Y_\beta \in \Gamma(T^* M)$ 
    satisfy Assumption \ref{ass:XY} below.
    Then there exists a metric 
        $g_{\alpha \beta} 
            \in \Gamma(T^* M \otimes T^* M )
        $ 
    satisfying  
        $Y_\beta = g_{\alpha \beta} X^\alpha$
    if and only if the following conditions hold:
    \begin{enumerate}[$(i)$]
    \item For all $m \in M$ with 
            $Y_\beta|_m \neq 0$
            we have
                $X^\alpha Y_\alpha|_m > 0$;            
    \item For all $m \in M$ with $Y_\beta|_m = 0$ 
            we have 
                $X^\alpha|_m = 0$;
    \item For all $m \in M$ with $Y_\beta|_m = 0$
            there exists a scalar product 
            $\bar g_{\alpha \beta} 
                \in T_m^* M \otimes_{\rm S} T_m^* M$
        such that
        \begin{align*}
            \nabla_\alpha Y_\gamma|_m
            =
            \bar g_{\beta \gamma} 
                \nabla_\alpha X^\beta|_m.
        \end{align*}
    \end{enumerate}
    \end{theorem}

The choice of the connection $\nabla$ in \emph{(iii)} is arbitrary. 

We shall also prove a variant of this result  where $X^\alpha$ and $Y_\beta$ are of class $C^{k+1}$ for some $k \in \N$. 
In this case, the metric 
    $g_{\alpha \beta}$
is of class $C^k$; see Theorem \ref{thm:Ck}
below.

While Theorem \ref{thm:main-intro} is of independent interest, 
our motivation comes from an open question
on gradient flow structures for 
dissipative quantum systems, that will be discussed below.

Let us first briefly sketch the structure of the proof.
To prove the sufficiency of conditions \emph{(i)--(iii)},
it suffices to construct a \emph{local} metric 
around every point of $M$. 
The global metric can then be constructed using a partition of unity.
Around non-critical points the construction is straightforward: 
in local coordinates, it corresponds to constructing a positive definite matrix that maps one given vector to another one.
However, it is \emph{not} trivial to construct a smooth metric satisfying  
    $Y_\beta = g_{\alpha \beta} X^\alpha$
in a neighbourhood of a critical point. 

To solve this problem, we assume that the sought metric has a power series expansion in a suitable chart around the critical point. 
We then derive an infinite hierarchy of tensor equations, 
which express power series coefficients of degree $N$  
in terms of coefficients of degree at most $N - 1$ for $N \geq 1$.
Solvability of the lowest order equation is guaranteed by compatibility condition \emph{(iii)}.
We then prove that higher order equations can be solved iteratively.
Moreover, the norms of the solutions are exponentially bounded in the degree, which allows us to construct a convergent power series that satisfies the desired equation in a neighbourhood of the critical point.

\subsection*{Application to gradient structures}

Consider now the special case where 
    $Y \in \Gamma(T^* M)$ 
is the derivative of a smooth function
    $f \in C^\infty$, 
i.e, 
    $Y_\beta = \nabla_\beta f$.
Then our question becomes: 
\emph{Does there exist a smooth Riemannian metric $g_{\alpha \beta}$ such that
$X$ is the gradient of $f$ with respect to the metric $g$, 
i.e.,
    $X^\alpha = g^{\alpha \beta} \nabla_\beta f$}?
In other words, the question is whether the 
ODE     
    $\dot u = -X(u)$ on $M$ 
can be formulated as a gradient flow equation
$
    \dot u(t) = - \nabla f\big(u(t)\big)
$
for a suitable Riemannian metric.
Our main result yields necessary and suffcient conditions.

Gradient flows describe motion in the direction of steepest descent of the function $f$ in the geometry defined by the metric $g$.
The identification of an ODE as a gradient flow equation is often fruitful, as there are powerful techniques available for the analysis of gradient flows 
\cite{Ambrosio-Gigli-Savare:2008}. 

As an application of our main result, we address an open question on the gradient flow structure of finite-dimensional dissipative quantum systems. 
To put this result into context, let us first discuss the corresponding classical setting.

\subsubsection*{Classical Markov semigroups}

Consider an irreducible continuous-time Markov chain 
on a finite set $\cX$ with transition rates 
    $q_{xy} \geq 0$ for $x, y \in \cX$ with $x \neq y$.
The associated Markov semigroup $(P_t)_{t \geq 0}$ is a $C_0$-semigroup of positive operators on $\R^\cX$ that preserves the constant functions.
Its infinitesimal generator 
    $L : \R^\cX \to \R^\cX$ 
is given by
\begin{align*}
    \big(L \psi\big)(x) := \sum_{y \in \cX} 
            q_{xy} \big( \psi(y) - \psi(x) \big).
\end{align*}
As time evolves, the marginal law of the Markov chain 
describes a curve $(\mu_t)_{t > 0}$ in 
$\cP_*(\cX)$, the simplex of probability densities with positive density. 
It evolves according to the Kolmogorov forward equation (KFE)
\begin{align*}
    \partial_t \mu_t = L^* \mu_t, 
    \quad \text{ where }
    \big(L^* \mu\big)(x) 
        = 
    \sum_{y \neq x} \mu(y) q_{yx} - \mu(x) q_{xy}
\end{align*}
for $\mu \in \cP(\cX)$.
Let 
    $\pi \in \cP_*(\cX)$ 
be the unique stationary distribution.
It is well known and easy to verify that the relative entropy
\begin{align*}
    \Ent_\pi(\mu) := \sum_{x \in \cX} 
                    \mu(x) \log \Big(\frac{\mu(x)}{\pi(x)}\Big)
\end{align*}
decreases along trajectories of the KFE.

Much more is true if the Markov chain is \emph{reversible}, i.e., the \emph{detailed balance} condition
$\pi_x q_{x y} = \pi_y q_{y x}$
holds for all $x \neq y$.
Equivalently, this means that the generator $L$ is selfadjoint in the Hilbert space 
    $L^2(\cX,\pi)$.
In this case, it was shown in  \cite{Maas:2011,Mielke:2011} that 
the KFE can be written as the gradient flow equation of $\Ent_\pi$ 
with respect to a Riemannian metric on $\cP_*(\cX)$. 
The associated Riemannian distance is given by a discrete dynamical optimal transport problem, in the spirit of the Benamou--Brenier formulation for the Wasserstein distance \cite{Benamou-Brenier:2000}. 
This gradient flow structure is a discrete version of the Wasserstein gradient flow structure for the Fokker--Planck equation discovered by Jordan, Kinderlehrer, and Otto \cite{Jordan-Kinderlehrer-Otto:1998}. 
This construction has been the starting point for the development of discrete Ricci curvature based on geodesic convexity with applications to functional inequalities 
\cite{Erbar-Maas:2012,Mielke:2013,Erbar-Maas:2014,Fathi-Maas:2016,Erbar-Henderson-Menz:2017}

It was shown by Dietert \cite{Dietert:2015} that the reversibility assumption is also necessary: 
if the KFE can be written as gradient flow equation for $\Ent_\pi$
with respect to \emph{some} Riemannian metric on $\cP_*(\cX)$, 
then the underlying Markov chain is necessarily reversible.
Combined with the results from \cite{Maas:2011,Mielke:2011}, 
this result characterises reversible Markov chains as exactly those that admit a gradient flow structure for the relative entropy $\Ent_\pi$.

In this paper we provide a noncommutative analogue of this result.

\subsubsection*{Quantum Markov semigroups}

Let $(\cP_t)_{t \geq 0}$ be a quantum Markov semigroup on a finite-dimensional $C^*$-algebra $\cA$, i.e., 
$(\cP_t)_{t \geq 0}$
is a $C_0$-semigroup of linear operators on $\cA$ such that
$\cP_t \one  = \one$  
and the operators
$\cP_t$ are completely positive, i.e., $\cP_t \otimes I_n$ is a positive operator on $\cA \otimes \M_n(\C)$ for all $n \geq 1$.
(Here, $\one \in \cA$ denotes the unit element, and $I_n$ denotes the identity operator on the algebra of $n \times n$-matrices $\M_n(\C)$.) 
The infinitesimal generator of $(\cP_t)_{t \geq 0}$ 
will be denoted by $\cL$.

Let $(\cP_t^\dagger)_{t \geq 0}$ be the adjoint semigroup with respect to the duality pairing 
    $\ip{A,B} = \Tr[A^* B]$. 
This is a $C_0$-semigroup of completely positive and trace-preserving linear operators with generator $\cL^\dagger$.
In particular, the operators $\cP_t^\dagger$ map the set of density matrices 
    $\Dens := \{ \rho \in \cA \ : \ 
                 \rho \geq 0 \text{ and }
                 \Tr[\rho] = 1 \}$
into itself.
Here we restrict our attention to the \emph{ergodic} setting: we assume that there exists a unique stationary state, i.e., a unique density matrix      
    $\sigma \in \Dens$
satisfying $\cL^\dagger \sigma = 0$.
We shall assume that $\sigma$ is invertible.

The non-commutative analogue of the KFE is the 
\emph{Lindblad equation}
    $\partial_t \rho_t = \cL^\dagger \rho_t$.
It is well known \cite{Spohn:1978,Spohn-Lebowitz:1978} that the 
    \emph{von Neumann relative entropy}
\begin{align*}
    H_\sigma(\rho) := \Tr[\rho (\log \rho - \log \sigma)]
\end{align*}
decreases along solutions to this equation.
Moreover, 
following the earlier works 
\cite{Carlen-Maas:2014,Mielke:2013a},
it was shown in \cite{Carlen-Maas:2017,Mittnenzweig-Mielke:2017} that the Lindblad equation   
    $\partial_t \rho = \cL^\dagger \rho$ 
can be written as gradient flow equation for $H_\sigma$ under the condition of \emph{\textsc{gns}-detailed balance}.
This condition means that the generator $\cL$ is selfadjoint with respect to the weighted $L^2$-type scalar product
\begin{align*}
    \ip{A,B}_\sigma^{\textsc{gns}} :=
        \Tr[ \sigma A^* B]
\end{align*}
named after Gelfand, Naimark, and Segal.
As in the discrete setting above, the associated Riemannian metric is related to a dynamical optimal transport problem.

It is now natural to ask whether the condition of \textsc{gns}-detailed balance is also necessary for the existence of a gradient flow structure for the von Neumann relative entropy.
However, it was shown in \cite{Carlen-Maas:2020} that a different symmetry condition is necessary, namely the condition of 
\emph{\textsc{bkm}-detailed balance}. 
This condition corresponds to the selfadjointness of $\cL$ with respect to another weighted $L^2$-type scalar product
\begin{align*}
    \ip{A,B}_\sigma^{\textsc{bkm}} :=
        \int_0^1 \Tr[ \sigma^{1-s} A^* \sigma^s B] \dd s,
\end{align*}
named after Bogoliubov, Kubo, and Mori.
As the condition of \textsc{bkm}-detailed balance is strictly weaker than \textsc{gns}-detailed balance \cite{Carlen-Maas:2020}, 
there was a gap between the known necessary and sufficient conditions. 
As an application of Theorem \ref{thm:main-intro} we prove the following result, which closes this gap.

\begin{theorem}
    \label{thm:QMS-intro}
    Let $\cL$ be the generator of an ergodic quantum Markov semigroup on a finite dimensional $C^*$-algebra $\cA$,
    and
    let $\sigma \in \Dens_+$ be its stationary state. 
    The following statements are equivalent:
    \begin{enumerate}[(1)]
        \item The operator $\cL$ is selfadjoint with respect to the \textsc{bkm} scalar product 
            $\ip{\cdot, \cdot}_{\sigma}^{\textsc{bkm}}$.
        \item There exists a Riemannian metric on the interior of $\Dens$ for which the Lindblad equation 
            $\dot \rho_t = \cL^\dagger \rho_t$
        is the gradient flow equation of the von Neumann relative entropy $H_\sigma$.
    \end{enumerate}
\end{theorem}

The implication \emph{(2) $\Rightarrow$ (1)} was proved in \cite[Theorem 2.9]{Carlen-Maas:2020}. The converse implication is new.

\subsection*{Structure of the paper}

Section \ref{sec:main-result} contains the main result and a reformulation of the result in the gradient case.
The proof of the main result is contained in Section \ref{sec:proof}, except for the construction of the local metric, which is presented in Section \ref{sec:critical}.
Section \ref{sec:continuous} 
deals with 
the construction of 
a metric of class $C^k$ under the assumption that the fields $X^\alpha$ and $Y_\beta$
are of class $C^{k+1}$.
The application to quantum Markov semigroups is contained in Section \ref{sec: QMS}.

\section{Main results}
\label{sec:main-result}

Let 
    $X^\alpha \in \Gamma(T M)$ 
be a vector field and 
    $Y_\beta \in \Gamma(T^* M)$ 
be a co-vector field on a smooth manifold $M$.
Let 
    $\sfN_Y := \{ m \in M \ : \ Y|_m = 0 \}$ 
be the set of critical points of $Y$. 

In the sequel we impose the following mild assumptions on the fields $X^\alpha$ and $Y_\beta$.

\begin{assumption}
    \label{ass:XY}
    \begin{enumerate}[$(i)$]
    \item (Non-degeneracy)
        \label{ass:nondegenerate}
        The bilinear form 
            $\nabla_\alpha Y_\beta|_m$ 
        is non-degenerate for all $m \in \sfN_Y$ for some 
        (equivalently, any) connection $\nabla$. 
    \item (Real analyticity)
        \label{ass:smooth}
    For all $m \in \sfN_Y$ there exists a neighbourhood 
        $U_m \ni m$, 
    an open set 
        $\Omega \subset \R^n$,
    and a coordinate chart $\phi_m : U_m \to \Omega$, such that the fields 
        $\tilde X^a 
            := X^a \circ \phi_m^{-1} 
            : \Omega \to \R$ 
            and
        $\tilde Y_a 
            := Y_a \circ \phi_m^{-1} 
            : \Omega \to \R$ 
    have a converging power series expansion
    around $\phi_m(m)$
    for all $a \in \{1, \ldots, n\}$. 
\end{enumerate}
\end{assumption}

\begin{remark}
    The choice of the connection in \emph{(i)} above is irrelevant, since the difference of two connections $\nabla$ and $\tilde \nabla$ satisfies 
        $\tilde \nabla_\alpha Y_\beta 
            -   \nabla_\alpha Y_\beta
            = \Gamma_{\alpha \beta}^\gamma  Y_\gamma$, 
    where 
        $\Gamma_{\alpha \beta}^\gamma$
    is a $(1,2)$ tensor.
    In particular, \
        $\tilde \nabla_\alpha Y_\beta 
        =   \nabla_\alpha Y_\beta$ 
    for 
        $m \in \sfN_Y$.
    For the same reason, the choice of the connection is irrelevant in \emph{(iii)} in the following result.
\end{remark}

Using the notation introduced above, we restate our main result (Theorem \ref{thm:main-intro}) for the convenience of the reader.

\begin{theorem}[Main result]
\label{Theorem: Main}
Let 
    $X^\alpha \in \Gamma(T M)$ 
and 
    $Y_\beta \in \Gamma(T^* M)$ 
satisfy Assumption \ref{ass:XY}.
Then there exists a smooth metric 
    $g_{\alpha \beta} \in \Gamma(T^* M \otimes T^* M )$ 
satisfying  
$Y_\beta = g_{\alpha \beta} X^\alpha$, 
if and only if the following conditions hold:
\begin{enumerate}[$(i)$]
\item $X^\alpha Y_\alpha|_m>0$ 
        for all $m \in M \setminus \sfN_Y$;
\item $X^\alpha|_m=0$ 
        for all $m \in \sfN_Y$;
\item For all $m \in \sfN_Y$  
        there exists a scalar product 
        $\bar g_{\alpha \beta} 
            \in 
        T_m^* M \otimes_{\rm S} T_m^* M$, 
    such that
    \begin{align*}
        \nabla_\alpha Y_\gamma|_m
            =
            \bar g_{\beta \gamma} 
                \nabla_\alpha X^\beta|_m,
    \end{align*}
where $\nabla_\alpha$ is an arbitrary connection.
\end{enumerate}
\end{theorem}

\begin{remark}
    As the necessity of the three conditions has been discussed above, it remains to prove their sufficiency.
    This will be done in Section \ref{sec:proof} below.
\end{remark}

In the special case where the co-vector field 
    $Y_\alpha := \nabla_\alpha F \in \Gamma(T^* M)$ 
is the derivative of a scalar function $f : M \to \R$, the above result admits a convenient reformulation.  
Assuming that $f$ attains its minimum at a unique critical point $\bar m \in M$,
the next results shows that property \emph{(iii)} above is 
equivalent to the symmetry and positivity of
the linearised map 
    $\Lambda: T_{\bar m} M \to T_{\bar m} M$, 
    $Z \mapsto \nabla_Z X$, 
at the critical point $\bar m$.
The relevant scalar product is given by the Hessian of $f$.

\begin{corollary}[Gradient case]
\label{Corollary: Main}
    Let 
        $f \in C^\infty(M)$ 
    be a function and 
        $X^\alpha \in \Gamma(TM)$ 
    be a vector field, such that 
        $X^\alpha$ and $Y_\alpha := \nabla_\alpha f$ 
    satisfy Assumption \ref{ass:XY}. 
    Suppose that $Y$ has a unique zero, $\bar m \in M$, 
    at which $f$ attains its  minimum.
    Then there exists a Riemannian metric
        $g_{\alpha \beta}
        \in \Gamma(T^* M \otimes T^* M )$
    satisfying    
    \begin{align*}
        \nabla_\beta f = g_{\alpha \beta} X^\alpha,
    \end{align*}
    if and only if the following conditions hold:
    \begin{enumerate}[(i)]
    \item $\nabla_{X^\alpha} f |_m < 0$ for all $m \in M$ with $m \neq \bar m$;
    \item $X^\alpha|_{\bar m} = 0$;
    \item The linear map 
        $\Lambda := \nabla_\alpha X^\beta|_{\bar m} 
                : 
            T_{\bar m} M \to T_{\bar m} M$ 
        is positive and symmetric with respect to the Hessian scalar product 
            $h_{\alpha \beta} := \nabla_\alpha \nabla_\beta f|_{\bar m}$ on $T_{\bar m} M$.
    \end{enumerate}
    \end{corollary}

    \begin{proof}
    It is clear that the conditions \emph{(i)} and \emph{(ii)} match the corresponding conditions in Theorem \ref{Theorem: Main}.

    Suppose now that condition \emph{(iii)} from Theorem \ref{Theorem: Main} holds, for some scalar product 
        $\bar g^{\alpha \beta} \in T_{\bar m} M \otimes_{\rm S} T_{\bar m} M$.
    We have to show that
    \begin{align*}
        h_{\alpha \beta}
            (\Lambda Z)^\alpha W^\beta
        & = h_{\alpha \beta}
            Z^\alpha (\Lambda W)^\beta
        && \text{for all } Z^\alpha, W^\alpha \in T_{\bar m} M, \text{ and } \\
        h_{\alpha \beta}
        (\Lambda Z)^\alpha Z^\beta 
        & > 0    
        && \text{for all }Z^\alpha \in T_{\bar m} M, \ 
        Z^\alpha \neq 0.
    \end{align*}
    To show this, note that
        $(\Lambda Z)^\alpha
            = Z^\gamma \nabla_\gamma X^\alpha
            = Z^\gamma \bar g^{\alpha \delta} 
                h_{\delta \gamma}$
    for $Z^\alpha \in T_{\bar m} M$.
    Hence, for $W^\alpha \in T_{\bar m} M$, 
    we see that the expression
    \begin{align*}
        h_{\alpha \beta}
            (\Lambda Z)^\alpha W^\beta
        =    
            h_{\alpha \beta}
            \bar g^{\alpha \delta} 
            h_{\delta \gamma}
                Z^\gamma W^\beta
    \end{align*}
    is invariant under interchanging $Z$ and $W$, which proves the desired symmetry.
    Moreover, this expression implies that
        $h_{\alpha \beta}
        (\Lambda Z)^\alpha Z^\beta
        = \bar g^{\alpha \beta} 
            \tilde Z_\alpha
            \tilde Z_\beta$
    where $\tilde Z_\alpha = h_{\alpha \beta} Z^\beta$.
    Since 
        $h_{\alpha \beta}$ 
    is invertible
    by Assumption \ref{ass:XY}
    and
        $\bar g^{\alpha \beta}$ 
    is positive definite, it follows that  
        $h_{\alpha \beta}
        (\Lambda Z)^\alpha Z^\beta
        > 0$ 
    whenever $Z^\alpha \neq 0$. 

    Conversely, suppose that condition \emph{(iii)} of the corollary holds. 
    For all $Z^\alpha, W^\alpha \in T_{\bar m} M$
    it follows that 
    $
        h_{\alpha \beta}
            (\Lambda Z)^\alpha W^\beta
        = \tilde g_{\alpha \beta} 
            Z^\alpha W^\beta
    $
    for a positive and symmetric tensor
    $\tilde g_{\alpha \beta} \in 
    T_{\bar m}^* M \otimes_{\rm S} T_{\bar m}^* M$.
    Since 
    $
        h_{\alpha \beta}
            (\Lambda Z)^\alpha W^\beta
        =     
        h_{\alpha \beta}
        Z^\gamma \nabla_\gamma X^\alpha  W^\beta
    $
    we infer that 
    $\tilde g_{\alpha \beta}
        = h_{\gamma \beta} \nabla_\alpha X^\gamma$.
    Now define 
    \begin{align*}
        \bar g^{\alpha \beta}
            := h^{\alpha \delta} 
            \tilde g_{\delta \gamma} 
                 h^{\gamma \beta}
        \in T_{\bar m} M \otimes T_{\bar m} M.
    \end{align*}
    Since $\tilde g_{\alpha \beta}$ is positive and symmetric and $ h^{\alpha \delta}$ is invertible, 
        $\bar g^{\alpha \beta}$
    defines a scalar product. Moreover, we have the desired identity
    $
        \nabla_\alpha X^\beta|_{\bar m} 
        = 
        \bar g^{\beta \gamma} h_{\alpha \gamma},
    $
    which completes the proof.
\end{proof}

In the special case were $Y_\beta$ is the derivative of a scalar function $f$, the existence of a metric 
satisfying 
        $\nabla_\beta f = g_{\alpha \beta} X^\alpha$
was proved in \cite{Barta-Chill-Fasangova:2012} on the complement of the set of critical points. 
The existence of a metric 
    with the desired property 
    on the whole manifold was stated as an open question \cite[Question 1]{Barta-Chill-Fasangova:2012}.
    Subsequently, under an additional assumption, which corresponds to $(iii)$ in Theorem \ref{Theorem: Main}, the existence of a \emph{continuous} extension of $g_{\alpha \beta}$ to all of $M$ was obtained in \cite{Bily:2014}; 
    cf.~Section \ref{sec:continuous} below for more details.
    However, the metric constructed \cite{Bily:2014} is 
    in general \emph{not} differentiable, even if the fields $X^\alpha$ and $Y_\beta$ are smooth; see Example \ref{ex:nondifferentiable} below.

    Here we show that $C^k$-regularity of the metric can be obtained if the fields $X^\alpha$ and $Y_\beta$ are assumed to be of class $C^{k+1}$.

    \begin{theorem}[Existence of a metric of class $C^k$]
        \label{thm:Ck}
        Let 
            $X^\alpha$ 
        and 
            $Y_\beta$
        be of class $C^{k+1}$ on $M$
        for some $k \in \N$
        and assume that  
            $\nabla_\alpha Y_\beta|_m$ 
        is non-degenerate for all $m \in \sfN_Y$ for some 
        (equivalently, any) connection $\nabla$.
        Then there exists a metric 
            $g_{\alpha \beta} 
            $
        of class $C^k$ on $M$
        satisfying  
        $Y_\beta = g_{\alpha \beta} X^\alpha$ 
        if and only if conditions $(i)$, $(ii)$, and $(iii)$ of Theorem \ref{Theorem: Main} hold.
        \end{theorem}

The proof of this result will be given in Section \ref{sec:continuous} below. 
It relies on the construction based on tensor equations that we develop in the proof of Theorem \ref{Theorem: Main}.

\section{Proof of the main result}
\label{sec:proof}

Our main result (Theorem \ref{Theorem: Main}) 
relies on two local versions of this result. 
First we construct a local solution around any non-critical point $m \in M \setminus \sfN_Y$.
In the special case were $Y_\beta$ is the derivative of a scalar function, a different construction of a metric away from critical points was carried out in \cite{Barta-Chill-Fasangova:2012}; see Section \ref{sec:continuous} below.

\begin{theorem}[Local solutions around non-critical points]
\label{Theorem: Local non-crit}
    Suppose that
        $X^\alpha \in \Gamma(T M)$ 
    and     
        $Y_\beta \in \Gamma(T^* M)$ 
    satisfy 
        $X^\alpha Y_\alpha|_{\bar m}>0$ 
    for some $\bar m \in M$.
    Then there exists a neighbourhood $U$ of $\bar m$ 
    and a smooth local metric 
        $g_{\alpha  \beta} : U \to T^*M \otimes T^*M$ 
    such that
    \begin{align}
    \label{eq:claim}
        X^\alpha|_m = g^{\alpha \beta} Y_\beta|_m
    \end{align}
    for all $m \in U$.
\end{theorem}

\begin{proof}
Since $X^\alpha Y_\alpha|_m > 0$, we have $Y_\alpha|_m \neq 0$. 
Therefore, we can complete the co-vector field 
    $e^1_\alpha := Y_\alpha \in T^* M$ to a dual frame 
    $E := ( e^1_\alpha, \ldots, e^n_\alpha )$ 
in a neighbourhood $V$ of $m$, 
i.e., $( e^1_\alpha|_m, \ldots, e^n_\alpha|_m )$ 
is a basis of $T_m^* M$ for all $m \in V$. 
The coordinates of $X^\alpha$ with respect to this frame 
are given by 
    $X^j := X^\alpha e^j_\alpha : V \to \R$
for $j = 1, \ldots, n$. 
Since $X^1|_{\bar m} > 0$, 
the set $U := V \cap \{ X^1 > 0 \}$ 
is still a neighbourhood of $\bar m$. 
Let us define 
    $\bar X : U \to \R^{n-1}$ 
and
    $f : U \to \R$ 
by
\begin{align*}
    \bar X := (X^2, \ldots, X^n), 
        \qquad
    f := \frac{X^1}{2} + \frac{2}{X^1} |\bar{X}|^2.
\end{align*}
We then define the bilinear form 
    $g^{\alpha \beta}$ 
in coordinates 
    $G = (g^{i j})_{i,j=1}^n$ as
\begin{align*}
    G :=
    \begin{bmatrix}
        X^1 & \bar{X}^\intercal\\
        \bar{X} & f I_{n-1}
    \end{bmatrix},
\end{align*}
where $I_n$ is the identity matrix.
Since the matrix $G$ is symmetric, the bilinear form $g$ is symmetric as well. 
To verify that $G > 0$, we write
\begin{align*}
G & = 
    \begin{bmatrix}
            \sqrt{\frac{X^1}{2}} 
            \\ \sqrt{\frac{2}{X^1}}\bar{X}
        \end{bmatrix}
    \begin{bmatrix}
            \sqrt{\frac{X^1}{2}} 
            & \sqrt{\frac{2}{X^1}}\bar{X}^\intercal
    \end{bmatrix}
    + 
    \begin{bmatrix}
        \frac{X^1}{2} & 0 \\ 
        0 & f I_{n-1} - \frac{2}{X^1} \bar{X} \bar{X}^\intercal 
    \end{bmatrix}
    \\ 
    & \geq 
    \begin{bmatrix}
       \frac{X^1}{2} & 0 \\ 
            0 & \big( f - \frac{2}{X^1} |\bar{X}|^2 \big) I_{n-1}
        \end{bmatrix}
    = 
    \frac{X^1}{2} I_{n} > 0,
\end{align*}
as desired.
To complete the proof,
note that the coordinates of $Y_\alpha$ are 
given by $Y_1 = 1$ and $Y_j = 0$ for $j \neq 1$.
Consequently,
\begin{align*}
    ( g^{\alpha \beta} Y_\beta )^i
    = \sum_j g^{i j} Y_j
    = g^{i 1}
    = X^i,
\end{align*}
which shows \eqref{eq:claim}.
\end{proof}

The second local version of Theorem \ref{Theorem: Main} concerns the construction of a smooth local metric in a neighbourhood of a critical point.

\begin{theorem}[Local solutions around critical points]
\label{Theorem: Local Crit}
Let 
        $X^\alpha \in \Gamma(T M)$ 
    and     
        $Y_\beta \in \Gamma(T^* M)$ 
    satisfy Assumption \ref{ass:XY}.
Suppose that
        $X^\alpha|_{\bar m} = Y_\alpha|_{\bar m} = 0$
for some $\bar m \in M$,
and suppose that 
there exists a scalar product 
        $\bar g \in T_{\bar m} M \otimes_{\rm S} T_{\bar m} M$, 
    such that
    \begin{align*}
        \nabla_\alpha X^\beta|_{\bar m} 
        = 
        \bar g^{\beta \gamma}
            \nabla_\alpha Y_\gamma|_{\bar m}.
    \end{align*}
Then there exists a neighbourhood $U$ of $m$ 
and a smooth local metric    
    $g_{\alpha \beta} : U \to T^*M \otimes T^*M$ 
such that
\begin{align*}
    X^\alpha|_m = g^{\alpha \beta} Y_\beta|_m
\end{align*}
    for all $m \in U$.
\end{theorem}

The proof of Theorem \ref{Theorem: Local Crit} is the main challenge of this paper and will be carried out in section \ref{sec:critical}.

We now show that the main result (Theorem \ref{Theorem: Main}) follows readily 
from the local Theorems 
    \ref{Theorem: Local non-crit} 
and  
    \ref{Theorem: Local Crit}
using a partition of unity argument; 
see, e.g., \cite[Theorem 1.131]{Gallot-Hulin-Lafontaine:2004} for the existence
of a partition of unity.

\begin{proof}[Proof of Theorem \ref{Theorem: Main}]
    The local results Theorems 
    \ref{Theorem: Local non-crit} 
    and  
    \ref{Theorem: Local Crit}
    guarantee that for any $m \in M$ 
    there exists a neighbourhood $U_m$ 
    and a local metric 
        $g_{\alpha \beta}$ 
    defined on $U_m$, 
    such that the desired identity
    \begin{align*}
        X^\alpha = g^{\alpha \beta} Y_\beta,
    \end{align*} 
    holds on $U_m$. 
    
    Let $\{ f_k \}_{k \in \N}$ 
    be a partition of unity subordinated to the cover 
        $\{ U_m : m \in M \}$ of the manifold $M$, 
    i.e.,
    there exists a locally finite open covering 
        $\{V_k\}_{k \in \N}$
    of $M$, such that each $V_k$ is contained in $U_{m_k}$ for some $m_k \in M$,    
    each function
    $f_k : M \to \R$ is nonnegative and smooth and 
    its support is contained in $V_k$, 
    and we have
    $
        \sum_{k \in \N} f_k(m) = 1
    $
    for all $m \in M$ 
    (where the sum is finite for each $m$).
    We then define 
    \begin{align*}
        g^{\alpha \beta} 
            :=
        \sum_{k \in \N}
            f_k  g_{m_k}^{\alpha \beta}.
    \end{align*}
    As $g^{\alpha \beta}$ is a  
    finite convex combination of the scalar products
        $g^{\alpha \beta}_{m_k}$, 
    it is a scalar product. 
    By linearity, $g^{\alpha \beta}$
    satisfies the desired equation 
        $X^\alpha = g^{\alpha \beta} Y_\beta$.
\end{proof}

\section{Local solutions around critical points}
\label{sec:critical}

In this section we give the proof of Theorem \ref{Theorem: Local Crit}, which deals with the construction of the metric around critical points.

Fix $\bar m \in M$ and let $\phi : U \to \Omega$ be a coordinate chart 
which maps a neighbourhood 
    $U$ of $\bar m$
onto an open set $\Omega \subseteq \R^n$.
Using this chart we can identify the vector field
    $X^\alpha \in \Gamma(T M)$
defined on $U \subseteq M$
with the function
    $\tilde X^\alpha : \Omega \to V := \R^n$,
where $\tilde X^\alpha := X^\alpha \circ \phi^{-1}$.
Similarly, the co-vector field
    $Y_\beta \in \Gamma(T^* M)$
defined on $U \subseteq M$ 
can be identified with a function
    $\tilde Y_\beta : \Omega \to V^*$,
and the metric 
    $g_{\alpha \beta} 
        \in \Gamma(T^* M \otimes_\rmS T^* M)$
can be identified with a function
    $\tilde g_{\alpha \beta} 
        : \Omega \to V^* \otimes_\rmS V^*$.
In the remainder of this section, we will work on a fixed chart and remove the tildes to lighten notation.

\subsection{Motivation of the tensor equations}

Let $\bar x \in \Omega$ be such that 
    $Y_\beta|_{\bar x} = 0$,
and suppose that the identity 
    $X^\alpha = g^{\alpha \beta} Y_\beta$
holds in a neighbourhood of $\bar x$.
For $N \in \N$ and all indices 
    $c_1, \ldots, c_N \in \{ 1, \ldots, n \}$
we will derive a system of equations that the partial derivatives 
    $T_{c_1 \cdots c_N}^{a b} 
        := 
    \partial_{c_1} \cdots \partial_{c_N} g^{ab}$
satisfy at $x = \bar x$.

Taking partial differentives 
    $\partial_c$ for $c \in \{1, \ldots, n\}$ 
yields
\begin{align*}
    \partial_c X^a 
        = 
    \partial_c g^{a b} Y_b
        +
    g^{a b} \partial_c Y_b.
\end{align*}
Since $Y_b|_{\bar x} = 0$, we find that 
\begin{align*}
    \partial_c X^a 
        = 
    g^{a b} \partial_c Y_b
\end{align*}
at $x = \bar x$.
Taking second order derivatives, we find, 
for $c_1, c_2 \in \{1, \ldots, n\}$,
\begin{align*}
    \partial_{c_1} \partial_{c_2} X^a 
            = 
    \partial_{c_1} \partial_{c_2}  g^{a b} Y_b
        +
    \partial_{c_1} g^{a b} \partial_{c_2} Y_b
        +
    \partial_{c_2} g^{a b} \partial_{c_1} Y_b
        +
    g^{a b} \partial_{c_1} \partial_{c_2} Y_b.
\end{align*}
As $Y_b|_{\bar x} = 0$, the first term on the right-hand side vanishes, 
and we infer that the tensor of first-order derivatives
    $T_c^{a b} := \partial_c g^{a b}$ 
is a solution to the system
\begin{align*}    
    U_{c_2 b} T_{c_1}^{a b}
    +   U_{c_1 b} T_{c_2}^{a b}
    =   R_{c_1 c_2}^a,
\end{align*}
where 
    $U_{a b} 
    := \partial_a Y_b
    $
and 
    $R_{c_1 c_2}^a 
    := \partial_{c_1} \partial_{c_2} X^a 
    - g^{a b} \partial_{c_1} \partial_{c_2} Y_b.
    $

More generally, for $N = 1, 2, \ldots$, we find
\begin{align*}
    \partial_{c_1} \cdots \partial_{c_N} 
        X^a 
    = \sum_{ S \subseteq [N] }
    \partial_{c_S}
        g^{a b}
    \partial_{c_{[N] \setminus S}} 
        Y_b,
\end{align*}
where we use the shorthand notation
    $\partial_{c_S} 
    =  
        \partial_{c_{i_1}} 
        \cdots 
        \partial_{c_{i_k}}$
for $S = \{i_1, \ldots, i_k\}
    \subseteq \{1, \ldots, N\}$
with $i_\mu \neq i_\nu$ for $\mu \neq \nu$.
Since $Y_b = 0$, the term with $|S| = N$ vanishes.
Thus, the derivatives of order $(N-1)$, given by
    $T_{c_1 \cdots c_{N-1}}^{a b} 
    := 
        \partial_{c_1} \cdots \partial_{c_{N-1}}
        g^{a b}$ 
solve the system
\begin{align}
\label{eq:tensor-N}
    \sum_{i = 1}^N 
    U_{c_i b} 
    T_{c_1 
    \cdots \widecheck{c_i} 
    \cdots c_N 
    }^{a b}
    =   R_{c_1 \cdots c_N}^a,
\end{align}
where 
    $U_{c b} := \partial_c Y_b$,
and
\begin{align*}
    R_{c_1 \cdots c_N}^a 
    := 
    \partial_{c_1} \cdots \partial_{c_N} 
    X^a 
    - 
    \sum_{ \substack{S \subseteq [N] \\ |S| < N-1 }}
    \partial_{c_S}
    g^{a b}
    \partial_{c_{[N] \setminus S}} 
    Y_b
\end{align*}
depends on (derivatives of) $X$ and $Y$, and on derivatives of $g$ of order at most $N-2$.
The notation $T_{c_1 
\cdots \widecheck{c_i} 
\cdots c_N 
}^{a b}
$ means that the index $c_i$ is removed.

The identity \eqref{eq:tensor-N} suggests an iterative scheme to construct a local solution $g^{\alpha \beta}$ to the equation 
    $X^\alpha = g^{\alpha \beta} Y_\beta$
around a critical point $\bar x \in U$
as a power series    
    \begin{align*}
    g^{a b}|_{\bar x} := 
        \sum_{N = 0}^\infty
            \frac{1}{N!}
                T^{a b}_{c_1 \cdots c_N}
                (x - \bar x)^{c_1} \cdots 
                (x - \bar x)^{c_N}
    \end{align*}
    with coefficients 
        $T^{\alpha \beta}_{\gamma_1 \cdots \gamma_N}
            \in 
        V^{\otimes_\rmS 2} \otimes (V^*)^{\otimes_\rmS N}$
The idea is to define, for $N = 0$, 
    $T^{a b} := \bar g^{a b}$, 
where 
    $\bar g
        \in 
        T_{\bar x}^* M 
            \otimes_\rmS 
        T_{\bar x}^* M$
is the scalar product satisfying
\begin{align*}
    \partial_c X^a|_{\bar x} 
        = \bar g^{a b} 
        \ \partial_c Y_b|_{\bar x},
\end{align*}
which exists by assumption.
Higher order Taylor coefficients 
    $T_{c_1 \ldots c_N}^{a b}$ 
are then constructed by iteratively solving a system of tensor equations of the form \eqref{eq:tensor-N}.
    
Section \ref{sec:tensor-eq} deals with the existence of a solution to these equations.
The construction and the convergence of the iterative scheme is contained in Section \ref{sec:scheme}.

\subsection{Solving the tensor equations}
\label{sec:tensor-eq}

We start by formulating an explicit solution to the tensor equation \eqref{eq:tensor-N} of order $N = 2$.

\begin{lemma}
\label{lem:tensor-equation-2}    
    Let $V$ be a finite-dimensional vector space, 
    and let
    $R_{\gamma \delta}^\alpha \in 
    V \otimes (V^* \otimes_\rmS V^*) $
    and 
        $U_{\alpha \beta}\in V^* \otimes V^*$
    be given.
    We assume that $U_{a \beta}$ is invertible with inverse
        $U^{\alpha \beta}\in V \otimes V$. 
    Then the tensor 
        $T_\gamma^{\alpha \beta} 
            \in 
            (V \otimes V) \otimes V^*$
    defined by
    \begin{align*}
        T_\gamma^{\alpha \beta}
            :=  
        \frac12
            \Big(
                U^{\beta \delta} R_{\gamma \delta}^\alpha
            +   U^{\alpha \delta} R_{\gamma \delta}^\beta
            -   U_{\gamma \gamma'} U^{\alpha \alpha'} U^{\beta \beta'} 
                R_{\alpha' \beta'}^{\gamma'}
            \Big)
    \end{align*}
    satisfies the equations 
        $T_\gamma^{\alpha \beta} = T_\gamma^{\beta \alpha}$ 
    and
    \begin{align}   
    \label{eq:first}
        U_{\delta \beta} T_\gamma^{\alpha \beta}
    +   U_{\gamma \beta} T_\delta^{\alpha \beta}
    =   R_{\gamma \delta}^\alpha.
    \end{align}
\end{lemma}

\begin{proof}
    The fact that $T_\gamma^{\alpha \beta} = T_\gamma^{\beta \alpha}$ follows readily from the definition. 
    To show that \eqref{eq:first} holds, note that by definition of $T$,
    \begin{align}
     \label{eq:proof-a}   
        2 U_{\delta \beta} T_\gamma^{\alpha \beta}
        & = 
        R_{\gamma \delta}^\alpha
        +
        U_{\delta \beta} U^{\alpha \epsilon} R_{\gamma \epsilon}^\beta
        -
        U_{\gamma \gamma'} U^{\alpha \alpha'} R_{\alpha' \delta}^{\gamma'}, \\
    \label{eq:proof-b}   
        2 U_{\gamma \beta} T_d^{\alpha \beta}
        & = 
        R_{\delta \gamma}^\alpha
        +
        U_{\gamma \beta} U^{\alpha \epsilon} R_{\delta \epsilon}^\beta
        -
        U_{\delta \delta'} U^{\alpha \alpha'} R_{\alpha' \gamma}^{\delta'}.
    \end{align}
    Relabeling indices on the right-hand side and using the symmetry of $R$, we observe that the second term in \eqref{eq:proof-a} equals the third term in \eqref{eq:proof-b}, and the second term in \eqref{eq:proof-b} equals the third term in \eqref{eq:proof-a}.
    Summing these identities, we thus obtain \eqref{eq:first}.
\end{proof}

We also need the following multilinear generalisation.

\begin{lemma}
\label{lem:tensor-equation-N}     
    Fix $N \geq 2$.
    Let $V$ be a finite-dimensional vector space, 
    and let
        $R_{\gamma_1 \cdots \gamma_N}^\alpha 
        \in 
        V
        \otimes 
        (V^*)^{\otimes_\rms N}$
    and 
        $U_{\alpha \beta} \in V^* \otimes V^*$
    be given.
    We assume that $U_{\alpha \beta}$ is invertible with inverse
        $U^{\alpha \beta} \in V \otimes V$. 
    Then the tensor 
        $T_{\gamma_1 \cdots \gamma_{N-1}}^{\alpha \beta}
            \in 
            V^{\otimes_\rms 2} 
            \otimes 
            (V^*)^{\otimes_\rms (N-1)} 
        $
    defined by
    \begin{align}
    \label{eq:T-def}
        T_{\gamma_1 \cdots \gamma_{N-1}}^{\alpha \beta}
            :=  
        \frac{1}{N}
            \bigg(
                U^{\beta \delta} R_{\delta \gamma_1 \cdots \gamma_{N-1}}^\alpha
            +   U^{\alpha \delta} R_{\delta \gamma_1 \cdots \gamma_{N-1}}^\beta
            -   \frac{1}{N-1} \sum_{i = 1}^{N - 1}
                    U_{\gamma_i \gamma_i'} 
                    U^{\alpha \alpha'} 
                    U^{\beta \beta'} 
                    R_{\alpha' \beta' \cnominus{i}}^{\gamma_i'}
            \bigg)
    \end{align}
    satisfies 
    \begin{align}   
    \label{eq:first-multi}
        \sum_{i = 1}^N 
            U_{\gamma_i \beta} 
            T_{\gamma_1 
                \cdots \widecheck{\gamma_i} 
                \cdots \gamma_N 
            }^{\alpha \beta}
        =   R_{\gamma_1 \cdots \gamma_N}^\alpha.
    \end{align}
\end{lemma}

\begin{proof}
    The fact that $T$ belongs to  
    $V^{\otimes_\rms 2}
        \otimes 
    (V^*)^{\otimes_\rms (N-1)}         
    $
    follows readily from the definition. 
    To show that \eqref{eq:first-multi} holds, note that 
    \begin{align*}
    &   \sum_{i=1}^N
            U_{\gamma_i \beta} 
            T_{ \cno{i} }^{\alpha \beta}
    =   \frac1N \sum_{i=1}^N
        \bigg\{
            U_{\gamma_i \beta} U^{\beta \delta} R_{\delta \cno{i}}^\alpha
        +   U_{\gamma_i \beta} U^{\alpha \delta} R_{\delta \cno{i}}^\beta
    \\& \qquad \qquad\qquad \qquad\qquad \qquad  \quad
        -   \frac{1}{N-1} \sum_{j : j \neq i}
            U_{\gamma_i \beta} U^{\beta \beta'} U_{\gamma_j \gamma_j'} U^{\alpha \alpha'}
            R_{\alpha' \beta' \cnoo{i}{j}}^{\gamma_j'}
        \bigg\}
    \\& =   \frac1N \sum_{i=1}^N
        \bigg\{
            R_{\gamma_1 \cdots \gamma_N}^\alpha
        +   U_{\gamma_i \beta} U^{\alpha \delta} R_{\delta \cno{i}}^\beta
        -   \frac{1}{N-1} \sum_{j : j \neq i}
            U_{\gamma_j \gamma_j'} U^{\alpha \alpha'}
            R_{\alpha' \cno{j}}^{\gamma_j'}
        \bigg\}.
    \end{align*}
    This yields the result, as the first term has the desired form, and the second term cancels against the third term, as can be seen by renaming indices $(\alpha', \gamma_j')$ into $(\delta,\beta)$.
\end{proof}

\subsection{Iterative construction of the power series \& Proof of Theorem \ref{Theorem: Local Crit}}
\label{sec:scheme}

We now place ourselves in the setting of Theorem \ref{Theorem: Local Crit}.
Thus, let
    $X^\alpha \in \Gamma(T M)$ 
and     
    $Y_\beta \in \Gamma(T^* M)$ 
satisfy Assumption \ref{ass:XY}, 
and suppose that
    $X^\alpha|_{\bar m} = Y_\alpha|_{\bar m} = 0$
for some fixed $\bar m \in M$.
We assume that 
there exists a scalar product 
    $\bar g \in T_{\bar m} M \otimes_{\rm S} T_{\bar m} M$
satisfying
\begin{align*}
    \nabla_\alpha X^\beta|_{\bar m} 
    = 
    \bar g^{\beta \gamma}
        \nabla_\alpha Y_\gamma|_{\bar m}.
\end{align*}
Our goal is to construct the local metric $g^{\alpha \beta}$ around 
    $\bar m$
as a convergent power series centered at $\bar x = \phi(\bar m)$.
We now present the definition of its coeffients
    $T_{c_1 \cdots c_N}^{a b}$,
which is motivated by the equations \eqref{eq:tensor-N}.
Our computations will be performed in a fixed chart 
    $\phi : U \to \Omega$ around $\bar m$ 
which satisfies Assumption \ref{ass:XY}.

\begin{definition}[The power series coeffients 
        $T_{c_1 \cdots c_N}^{a b}$]\mbox{}
\label{def:power-series}
Write 
    $U_{\alpha \beta} := \nabla_\alpha Y_\beta|_{\bar m}$
for brevity. 
\begin{itemize}
    \item 
\emph{Initialisation:}
We define the initial tensor $T^{\alpha \beta} \in 
    V \otimes_\rmS V$ 
of our iteration as
\begin{align*}
    T^{a b} := \bar g^{a b}.
\end{align*}

\item
\emph{Iterative step (special case $N=2$):}
We first define 
    $R_{\gamma \delta}^\alpha \in 
    V \otimes (V^* \otimes_\rmS V^*)$ 
by
\begin{align*}
    R_{c d}^a 
    := \partial_{c} \partial_{d} X^a 
    - T^{a b} \partial_{c} \partial_{d} Y_b
\end{align*}
and then define 
    $T_\gamma^{\alpha \beta} \in 
    (V \otimes_\rmS V) \otimes V^*$ 
as the solution to the system
\begin{align*}
    U_{d b} T_c^{a b}
+   U_{c b} T_d^{a b}
=   R_{c d}^a
\end{align*}
constructed in Lemma \ref{lem:tensor-equation-2}.

\item
\emph{Iterative step ($N = 2, 3, \ldots $):}
We first define
    $R_{\gamma_1 \cdots \gamma_N}^{\alpha} 
    \in 
    V \otimes (V^*)^{\otimes_\rmS N}$
in terms of the lower order tensors
$T^{\alpha \beta}, 
T_{\gamma_1}^{\alpha \beta}, 
\ldots, 
T_{\gamma_1 \cdots \gamma_{N-2}}^{\alpha \beta}$
by 
\begin{align}
    \label{eq:R-def}
    R_{c_1 \cdots c_N}^{a}
        & :=
    \partial_{c_1} \cdots \partial_{c_N} 
    X^a 
    - 
    \sum_{ \substack{S \subseteq [N] \\ |S| < N-1 }}
    T_{c_S}^{a b} \ 
    \partial_{c_{[N] \setminus S}} 
    Y_b.
\end{align}
Here we use the shorthand notation
$
    T_{c_S} 
        :=
    T_{c_{i_1} \cdots c_{i_k}}
$
for $S := \{ i_1, \ldots, i_k\}$
with $i_\mu \neq i_\nu$ for $\mu \neq \nu$.
Then we define the tensor 
    $T_{\gamma_1 \cdots \gamma_{N-1}}^{\alpha \beta} 
        \in 
    V^{\otimes_\rmS (N-1)} \otimes (V^*)^{\otimes_\rmS 2}
    $ 
as the solution to the system
\begin{align*}
    \sum_{i = 1}^N 
    U_{c_i b} 
    T_{c_1 
    \cdots \widecheck{c_i} 
    \cdots c_N 
    }^{a b}
    =   R_{c_1 \cdots c_N}^a,
\end{align*}
constructed in Lemma \ref{lem:tensor-equation-N}.
\end{itemize}
\end{definition}

\begin{remark}
    The nondegeneracy assumption on the derivative 
        $\nabla_\alpha Y_\beta|_{\bar m}$
    is crucially used in this construction, 
    as the application of Lemmas \ref{lem:tensor-equation-2} and \ref{lem:tensor-equation-N} requires the invertibility of 
        $U_{\alpha \beta}$.
\end{remark}

Our next aim is to show that the power series
\begin{align*}
    g^{a b}|_x := 
    \sum_{N = 0}^\infty
        \frac{1}{N!}
            T^{a b}_{c_1 \cdots c_N}
            (x - \bar x)^{c_1} \cdots 
            (x - \bar x)^{c_N}.
\end{align*}
converges and defines a Riemannian metric in a neigbourhood of $\bar x$.
For this purpose we equip the spaces
    $ V^{\otimes k} \otimes (V^*)^{\otimes \ell}$ 
with the norm
\begin{align*}
    \big\| 
        W_{\alpha_1 \cdots \alpha_k}
            ^{\beta_1 \ldots \beta_\ell} 
    \big\|_\infty
        :=\max_{ a_1, \ldots, a_k, b_1, \ldots, b_\ell}
            \big| W_{a_1 \cdots a_k}^{b_1 \cdots b_\ell} \big|,
\end{align*}
where $W_{a_1 \cdots a_k}^{b_1 \cdots b_\ell}$ are the coordinates of 
    $W_{\alpha_1 \cdots \alpha_k}
        ^{\beta_1 \ldots \beta_\ell}$ 
in the standard basis of $\R^n$.
For brevity, let us write
\begin{align*}
    r_N := \| R_{\gamma_1 \cdots \gamma_N}^\alpha \|_\infty
    \tand
    t_N := \| T_{\gamma_1 \cdots \gamma_N}^{\alpha \beta} \|_\infty.
\end{align*}
We then obtain the following crucial growth bound on the power series coefficients.

\begin{lemma}
\label{lem:power-series-bound}
    There exist constants $C, p < \infty$ such that 
        $t_N \leq C N! p^N$ for all $N \geq 1$.
\end{lemma}

\begin{proof} 
    Recall that we work in a chart for which 
    Assumption \ref{ass:XY} holds. 
    Therefore, the real analyticity assumption 
    implies that there exist constants
        $C', q < \infty$
    such that 
    \begin{align}
        \label{eq:powerseries-bounds}
        \big| 
            \partial_{c_1} 
                \cdots \partial_{c_M} 
            \tilde X^a|_{\bar x} 
        \big|
             \leq 
        C' M! q^M
    \tand
        \big| 
            \partial_{c_1} 
                \cdots \partial_{c_M} 
            \tilde Y_a|_{\bar x} 
        \big|
            & \leq 
        C' M! q^M
    \end{align}
    for all $m \in \N$ and 
    all $c_1, \ldots, c_M \in \{1, \ldots, n\}$; 
    see, e.g., \cite[Proposition 2.2.10]{Krantz-Parks:2002}.

    Since $U_{\alpha \beta}$ is non-degenerate by Assumption \ref{ass:XY}, we have
    \begin{align*}
        K := \max \big\{ 
                \|U_{\alpha \beta}\|_\infty, 
                \|U^{\alpha \beta}\|_\infty 
                    \big\}
        < \infty.
    \end{align*}
    Using the bounds on the power series coefficients from 
    \eqref{eq:powerseries-bounds}
    and the definitions of $T$ and $R$ from 
    \eqref{eq:T-def} and \eqref{eq:R-def},
    we 
    obtain the following relations between the norms $r_k$ and $t_k$: 
    \begin{align*}
    \frac{r_N}{N!}
    & \leq
        C' q^N 
        +
        \frac{C' n}{N!} 
            \sum_{ \substack{S \subseteq [N] \\ 
                    |S| < N - 1 }}
        t_{|S|} q^{N-|S|} \big( N - |S|\big)!
    \\ & =
        C' q^N 
        +
        \frac{C' n}{N!} \sum_{k = 0}^{N - 2}
        \binom{N}{k}
        t_{k}  
       q^{N-k} ( N - k)!
    =
        C' q^N 
        \bigg(
        1 
        +
        n \sum_{k = 0}^{N - 2}
            \frac{t_{k} }{k! q^{k}}
       \bigg)
    \end{align*}
    and 
    \begin{align*}
        t_{N-1}
        \leq 
        \frac{1}{N}
            \Big(
                2 n K r_N 
                +
                K^3 n^3 r_N
            \Big)
        =: \frac{\tilde K}{N} r_N,
    \end{align*}
    where $\tilde K < \infty$ depends on $K$ and $n$.
    Using these estimates we shall now prove the desired result by induction.  
    
    We thus assume, for some $N \geq 0$, that the desired inequality 
        $t_k / k! \leq C  p^k$ 
    holds for all $k \leq N$, with suitable constants 
        $C, p < \infty$.
    We will now show that $t_{N+1} / (N+1)!  \leq C p^{N+1}$. 
    Indeed, using the inequalities above and the induction assumption, we obtain
    \begin{align*}
        \frac{t_{N + 1}}{(N+1)!}
            \leq
        \frac{\tilde K}{(N + 2)!} r_{N + 2}
            \leq
        C' \tilde K
            q^{N + 2}
                \bigg(
                    1 + n \sum_{k=0}^N   
                        \frac{t_k}{k! q^k}
                \bigg)
            \leq
        C' \tilde K
            q^{N + 2}
                \bigg(
                    1 + C n \sum_{k=0}^N   
                        \Big(\frac{p}{q}\Big)^k
                \bigg).
    \end{align*}
    Assuming, without loss of generality, that 
        $C \geq 1$ and $p > q$,
    this yields
    \begin{align*}
        \frac{t_{N + 1}}{(N+1)!}
            & \leq
            C p^{N+1} 
            C' \tilde K q
            \bigg(
                \Big(\frac{q}{p}\Big)^{N+1} 
                    + n \sum_{k=0}^N   
                    \Big(\frac{q}{p}\Big)^{N-k+1} 
            \bigg)
            \leq
            C p^{N+1} 
            C' \tilde K q
            \bigg(
                \frac{q}{p}
                    + \frac{n q}{p - q}
            \bigg).
    \end{align*}
   By choosing $p$ sufficiently large, the last term in brackets can be made smaller than $(C' \tilde K q)^{-1}$. This yields the result.
\end{proof}

\begin{corollary}
\label{cor: construction}
There exists a neigbourhood $U \ni \bar x$, such that
the power series
\begin{align}
\label{eq:power-series}
    g^{a b}|_x := 
    \sum_{N = 0}^\infty
        \frac{1}{N!}
            T^{a b}_{c_1 \cdots c_N}
            (x - \bar x)^{c_1} \cdots 
            (x - \bar x)^{c_N}
\end{align}
converges for all $x \in U$, 
its inverse defines a Riemannian metric, and the equality $X^\alpha|_x = g^{\alpha \beta} Y_\beta|_x$ holds for all $x \in U$.
\end{corollary}

\begin{proof}
The definitions yield
\begin{align*}
        \big|    
            T^{a b}_{c_1 \cdots c_N}
            (x - \bar x)^{c_1} \cdots 
            (x - \bar x)^{c_N}
        \big|
    \leq 
        n^N
        \| 
            T^{\alpha \beta}_{\gamma_1 \cdots \gamma_N} 
        \|_\infty
    \| x - \bar x \|_1^N,
\end{align*}
where $\| y \|_1 := \sum_a | y^a |$ for $y \in V^*$.
Since 
    $\| 
        T_{\gamma_1 \cdots \gamma_N}^{\alpha \beta} 
    \|_\infty 
    \leq C N! p^N$
by Lemma \ref{lem:power-series-bound}, 
we infer that the power series \eqref{eq:power-series} converges for $\| x - \bar x\|_1 < 1/(p n)$.

To verify that $g^{\alpha \beta}$ defines a metric, 
note first that 
    $g^{a b} = g^{b a}$ 
by construction. 
To show that $g^{\alpha \beta}$ is positive definite when 
    $x$ is close enough to $\bar x$, 
it suffices to note that 
    $g^{\alpha \beta}|_{\bar x} = \bar g^{\alpha \beta}$ 
is positive definite
and the map $x \mapsto g^{\alpha \beta}|_{x}$ is continuous. 

Since the tensor fields 
    $X^\alpha$, $Y_\beta$, and $g^{\alpha \beta}$ 
are given by convergent power series, and since 
    $X^\alpha|_{\bar x}  
        = g^{\alpha \beta} Y_\beta |_{\bar x}
    $ 
by assumption, it is enough to verify that all derivatives at $\bar x$ coincide, i.e.,
\begin{align*}
    \partial_{c_1} \cdots \partial_{c_N} 
        X^a 
    =
    \partial_{c_1} \cdots \partial_{c_N} 
        (g^{a b} Y_b )
\end{align*}
for all  $N\in \N$ and all 
    $c_1, \ldots, c_N \in \{ 1, \ldots, n \}$. 
To prove this identity,
we use the notation from Definition \ref{def:power-series},
to obtain at 
$x = \bar x$,
\begin{equation}
\begin{aligned}
    \label{eq:derivatives-identity}
    \partial_{c_1} \cdots \partial_{c_N} 
        (g^{a b} Y_b )
    &
     =
    \sum_{ S \subseteq [N] }
    \partial_{c_S}
        g^{a b}
    \partial_{c_{[N] \setminus S}} 
        Y_b
    \\ &
    =
    \Big(\partial_{c_1} \cdots \partial_{c_N} 
    g^{a b}
    \Big) Y_b
    +
    \sum_{i=1}^N
        \Big(
        \partial_{c_1} \cdots 
        \partial_{c_{i-1}} 
        \partial_{c_{i+1}} \cdots \partial_{c_N}     
        g^{a b}
        \Big)
        \partial_{c_i} Y_b 
\\ & \qquad
    +
    \sum_{ \substack{S \subseteq [N] \\ |S| < N-1 }}
    \partial_{c_S}
    g^{a b}
    \partial_{c_{[N] \setminus S}} 
    Y_b
    \\ & = 
    0 + 
    \sum_{i = 1}^N 
    U_{c_i b} 
    T_{c_1 
    \cdots \widecheck{c_i} 
    \cdots c_N 
    }^{a b}
    +
    \Big(
        \partial_{c_1} \cdots \partial_{c_N} X^a 
        - 
        R_{c_1 \cdots c_N}^a 
    \Big)
    \\ & =  \partial_{c_1} \cdots \partial_{c_N} X^a.
\end{aligned}    
\end{equation}
To obtain the third equality, we use that $\bar x$ is a critical point,
together with the definitions of $R$, $T$, and $U$ in Definition \ref{def:power-series}.
In the final step we use the tensor equation \eqref{eq:first-multi}.
\end{proof}

The proof of Theorem \ref{Theorem: Local Crit} is now complete, 
as the metric $g^{\alpha \beta}$ constructed above can be pushed back to $M$ using the chart $\phi$.

\section{Construction of a
    metric of class 
    \texorpdfstring{$C^k$}{Ck}
    }
\label{sec:continuous}

Let 
    $X^\alpha$ 
be a vector field and 
    $Y_\beta$ 
be a co-vector field on a smooth manifold $M$.
As before, let 
    $\sfN_Y := \{ m \in M \ : \ Y|_m = 0 \}$ 
be the set of critical points of $Y$. 
In this section we weaken the regularity assumptions on $X$ and $Y$. 
In Proposition \ref{prop:cont} these fields are assumed to be merely differentiable.  
Subsequently we provide the proof of Theorem \ref{thm:Ck}, which deals with fields of class $C^{k+1}$ for $k \in \N$.

The following result, which does not require an iterative scheme, is known in the special case where $Y_\beta$ is the derivative of a scalar function 
\cite{Barta-Chill-Fasangova:2012, Bily:2014}.
In this setting, the existence of a metric with the desired property away from critical points is proved in \cite{Barta-Chill-Fasangova:2012}.
The construction of the metric below is taken from there. 
It relies on the unique decomposition of vector fields into a component parallel to $X$ and a component annihilating $Y$, which only works away from critical points.
The proof of the existence of a continuous extension to all of $M$ is adapted from \cite{Bily:2014}.

\begin{proposition}[Existence of a continuous metric]
    \label{prop:cont}
    Let 
        $X^\alpha$ 
    and 
        $Y_\beta$
    be differentiable fields on $M$
    and suppose that the bilinear form 
        $\nabla_\alpha Y_\beta|_m$ 
    is non-degenerate for all $m \in \sfN_Y$ for some 
    (equivalently, any) connection $\nabla$. 
    Suppose that the following conditions hold:
    \begin{enumerate}[$(i)$]
    \item $X^\alpha Y_\alpha|_m>0$ 
            for all $m \in M \setminus \sfN_Y$;
    \item $X^\alpha|_m=0$ 
            for all $m \in \sfN_Y$;
    \item For all $m \in \sfN_Y$  
            there exists a scalar product 
            $\bar g_m \in T_m M \otimes_{\rm S} T_m M$, 
        such that
        \begin{align*}
            \nabla_\alpha Y_\gamma|_m
                =
                \bar g_{\beta \gamma} 
                    \nabla_\alpha X^\beta|_m,
        \end{align*}
    where $\nabla_\alpha$ is an arbitrary connection.
    \end{enumerate}
    Then there exists a continuous metric $g_{\alpha \beta}$ 
    on $M$
    satisfying  
    $Y_\beta = g_{\alpha \beta} X^\alpha$. 
    \end{proposition}

    \begin{proof}
Let $m \in M \setminus \sfN_Y$ be a  non-critical point, hence $X|_m \neq 0$ and $Y|_m \neq 0$ by $(ii)$.
The assumption $(i)$ implies that we have the direct sum decomposition
    $T_m M = Y_m^\perp \oplus \Span\{X_m\}$,
hence every vector $Z \in T_m M$ can be uniquely decomposed as 
\begin{align*}
    Z = Z^{(0)} + Z^{(1)}, 
    \quad \text{with} \quad
    Z^{(0)} \in Y_m^\perp
        \tand
    Z^{(1)} := 
        \frac{\ip{Z,Y_m}}{\ip{X_m, Y_m}} X_m 
            \in 
        \Span\{X_m\}.
\end{align*}
Let $g = g_{\alpha \beta}$ be an arbitrary continuous metric on $M$ satisfying $g|_m = \bar g_m$ at all critical points $m \in \sfN_Y$.
Following \cite{Barta-Chill-Fasangova:2012}, we construct a perturbation of $\tilde g$ as follows:
\begin{align}
    \tilde g(Z,W)
        :=
    g(Z^{(0)},W^{(0)})  
        +
        \frac{
            \ip{Z^{(1)},Y}
            \;
            \ip{W^{(1)},Y}
            }
        {\ip{X, Y}},
\end{align}
for $Z, W \in \Gamma(TM)$.
In view of $(i)$, it readily follows that $g$ defines a continuous metric on $M \setminus \sfN_Y$. It remains to show that $\tilde g$ can be continuously extended to all of $M$.

It will be convenient to use abstract index notation. 
Taking into account that 
    $\ip{Z^{(1)},Y} = \ip{Z,Y}$ 
and $\ip{W^{(1)},Y} = \ip{W,Y}$,
it follows from the definition that
\begin{align*}
    \tilde g_{\alpha \beta}
    & = 
    g_{\alpha \beta}
    +
        \frac{Y_\alpha Y_\beta
            - g_{\alpha \gamma} 
                    X^\gamma Y_\beta
              - g_{\gamma \beta} 
                    X^\gamma Y_\alpha}
             {X^\delta Y_\delta}
    + 
    \frac{g_{\gamma \delta}
    X^\gamma X^\delta Y_\alpha Y_\beta}
         {(X^\delta Y_\delta)^2}.
\end{align*}
Introducing the deficit
    $R_\beta := Y_\beta 
        - g_{\alpha \beta} X^\alpha$,
we can write
\begin{align}
    \label{eq:tilde-g}
    \tilde g_{\alpha \beta}
    & = 
    g_{\alpha \beta}
    + 
        \frac{R_\alpha Y_\beta
        + R_\beta Y_\alpha}
             {X^\delta Y_\delta}
    - 
    \frac{R_\gamma X^\gamma  Y_\alpha Y_\beta}
         {(X^\delta Y_\delta)^2}.
\end{align} 

Fix a critical point $\bar m \in \sfN_Y$.
Using assumptions $(ii)$ and $(iii)$ we shall show that 
    $\tilde g|_m \to g|_{\bar m}$ 
as $m \to \bar m$, following the arguments in \cite{Bily:2014}.
Using the notation from Section \ref{sec:critical}, we shall perform a Taylor expansion of the terms in \eqref{eq:tilde-g} in a fixed chart, where $\bar m \in M$ corresponds to $\bar x \in \R^n$. 
As $X$ and $Y$ are differentiable, and $\bar x$ is a critical point, it follows from $(ii)$ that 
\begin{align}
    \label{eq:XY-expansions}
    X^a(x) 
        = 
    \nabla_c X^a(\bar x) (x - \bar x)^c 
        + o \big(|x-\bar x|\big)
    \tand
    Y_b(x) 
        = 
    \nabla_c Y_b(\bar x) (x - \bar x)^c 
        + o \big(|x-\bar x|\big).
\end{align}
Since $\bar g_{a b}(\bar x)$ is a scalar product, there exists $\kappa > 0$ such that
$\bar g_{a b}(\bar x) v^a v^b \geq \kappa |v|^2$ 
for all $v \in \R^n$.
Furthermore, $\nabla_b X^a$ is non-degenerate by assumption $(iii)$ and the non-degeneracy assumption on $\nabla_b Y^a$. 
Therefore, 
    $|\nabla_b X^a v |^2 
        \geq 
    \tilde \kappa | v |^2$
for some constant $\tilde \kappa > 0$.
Using these inequalities, together with $(iii)$, yields
\begin{equation}
    \begin{aligned}
\label{eq:XY-eq}
    X^a Y_a(x)
        & = 
            \nabla_b X^a(\bar x)
            \nabla_c Y_a(\bar x)
                (x - \bar x)^b 
                (x - \bar x)^c  
            + o \big(|x-\bar x|^2\big)
        \\ & = 
            \bar g_{a d}(\bar x) 
            \nabla_b X^a(\bar x)
            \nabla_c X^d(\bar x)
                (x - \bar x)^b 
                (x - \bar x)^c  
            + o \big(|x-\bar x|^2\big)
      \\ & \geq \kappa
            \big|
                \nabla_b X^a(\bar x)
                (x - \bar x)^b 
            \big|^2
            + o \big(|x-\bar x|^2\big)
      \\ &  \geq \kappa \tilde \kappa
            |
                x - \bar x
            |^2
            + o \big(|x-\bar x|^2\big),
    \end{aligned}
\end{equation}
which bounds the denominator in \eqref{eq:tilde-g} from below.
As for the terms in the numerator, we first note that
    $X^a(x) = O\big(|x-\bar x|\big)$ 
        and 
    $Y_b(x) = O\big(|x-\bar x|\big)$.
These bounds trivially imply that 
    $R_b(x) = O\big(|x-\bar x|\big)$ 
as well,
but this is not sufficient.
The key point of the proof is that this bound can be improved.
Indeed, using $(iii)$ and the continuity of $g$ at $\bar x$, we obtain
\begin{equation}
    \begin{aligned}
        \label{eq:R-eq}
    R_b(x) 
        & = 
    \big(Y_b - g_{a b} X^a\big)(x)
        \\& =
    \nabla_c Y_b(\bar x) (x - \bar x)^c 
    - g_{a b}(x)
    \nabla_c X^a(\bar x) (x - \bar x)^c 
        + o (|x-\bar x|)
        \\& =
    \big(\bar g_{a b}(x) 
        - g_{a b}(x)
       \big) 
    \nabla_c X^a(\bar x) (x - \bar x)^c 
            + o (|x-\bar x|)
    \\& =         o \big(|x-\bar x|\big).
\end{aligned}
\end{equation}
It now follows from \eqref{eq:XY-eq} and \eqref{eq:R-eq} together with the bounds on $X$ and $Y$, that the fractions in \eqref{eq:tilde-g} vanish as $x \to \bar x$.
This shows that 
$\tilde g$ can be continuously extended to $M$ by setting $\tilde g_{a b}(\bar x) := \bar g_{a b}(\bar x)$.
\end{proof}

While the metric $\widetilde{g}$ constructed in the proof of Proposition \ref{prop:cont} is continuous, it is \emph{not} in general differentiable, 
even if the background metric 
    $g_{\alpha \beta}$ 
and the vector fields
    $X^\alpha$ and $Y_\beta$ 
are smooth. 
Here is an explicit counterexample.

\begin{example}
    \label{ex:nondifferentiable}
Let $M$ be the open unit ball in $\R^2$.
We work in cartesian coordinates.
Set $X(x) = Y(x) = x$ for $x \in M$,
and consider the background metric 
    $g_{\alpha \beta}$ defined by 
\begin{align*}
g_{a b}(x):= 
    \begin{bmatrix}
       1 +x_2  & 0 \\ 
                0   & 1
    \end{bmatrix}
\end{align*}
for $x = (x_1, x_2) \in M$.
Since $g$ is smooth and $g|_0 = I$,
it is a valid background metric. 
An explicit computation yields 
\begin{align*}
    \widetilde{g}_{1 1}(x) 
    = 1 + \frac{x_2^5}{(x_1^2 + x_2^2)^2}
    \tand 
    \nabla_1 \widetilde{g}_{1 1}(x) 
    = 
    -4\frac{ x_1  x_2^5}{(x_1^2 + x_2^2)^3}.
\end{align*}
The latter is a non-constant homogeneous function and as such discontinuous at 
    $x = 0$, 
thus
    $\tilde g_{\alpha \beta}$ does not belong to $C^1$. 
\end{example}

Theorem \ref{thm:Ck} shows that 
better regularity properties can be obtained
by a careful choice
of the background metric $g_{\alpha \beta}$. 
In the following proof we define $g_{\alpha \beta}$
by making use of the construction in Section \ref{sec:critical}, 
which yields improved bounds on the deficit 
    $R_\beta := Y_\beta 
    - g_{\alpha \beta} X^\alpha$
around critical points.
This allows us to construct a metric 
    $\widetilde{g}_{\alpha \beta}$ 
of class $C^k$
whenever $X^\alpha$ and $Y_\beta$ are of class $C^{k+1}$.

\begin{proof}[Proof of Theorem \ref{thm:Ck}]
First we note that the necessity of conditions $(i)$ and $(ii)$ was already observed in the introduction. 
The necessity of $(iii)$ follows, even when $g$ is assumed to be merely continuous, 
from the expansions for $X$ and $Y$ in \eqref{eq:XY-expansions} and the expansion
    $g(x) = g(\bar x) + o\big(|x-\bar x|\big)$
in local coordinates around a critical point $\bar x$.
Therefore it remains to show that these three conditions are also sufficient.

As in  Proposition \ref{prop:cont}, 
we construct a metric 
of the form \eqref{eq:tilde-g}
on the non-critical set 
    $M \setminus \sfN_Y$:
\begin{align}
    \label{eq:g-formula}
    \tilde g_{\alpha \beta}
    & = 
    g_{\alpha \beta}
    + 
        \frac{R_\alpha Y_\beta
        + R_\beta Y_\alpha}
             {X^\delta Y_\delta}
    - 
    \frac{R_\gamma X^\gamma  Y_\alpha Y_\beta}
         {(X^\delta Y_\delta)^2},
\end{align}
where $R_\beta := Y_\beta 
- g_{\alpha \beta} X^\alpha$ denotes the deficit, 
and $g_{\alpha \beta}$ 
is a background metric on $M$
that will be carefully chosen below.
As noted before, it is immediate to verify that the desired identity 
    $Y_\beta = \widetilde{g}_{\alpha \beta} X^\alpha$ 
holds on $M \setminus \sfN_Y$.  

\smallskip

\emph{Construction of the background metric.} \
Fix $\bar{m} \in \sfN_Y$. 
As in Section \ref{sec:critical} we work in a fixed coordinate chart 
where $\bar{m}$ corresponds to $\bar{x} \in \R^n$.
In these local coordinates we then define the
background metric by
\begin{align*}
    g_{\bar{m}}^{a b}(x) 
        := 
    \sum_{N = 0}^k
        \frac{1}{N!}
            T^{a b}_{c_1 \cdots c_N}
            (x - \bar{x})^{c_1} \cdots 
            (x - \bar{x})^{c_N}
\end{align*}
for $x$ in a small neigbourhood around $\bar{x}$.
It is crucial that we use the tensors 
    $T^{\alpha \beta}_{\gamma_1 \cdots \gamma_N}$ 
that were constructed in Definition \ref{def:power-series}.
Note that $T^{\alpha \beta}_{\gamma_1 \cdots \gamma_N}$ is indeed well defined for $N\leq k$ due to our assumption that $X^\alpha$ and $Y_\beta$ are $k+1$ times continuously differentiable. 
As 
    $T^{\alpha \beta}$
is positive definite, 
it follows that 
    $(g_{\bar m})_{\alpha \beta}$
defines a metric in a neighbourhood of $\bar x$.

For each cricitical point $\bar m$, this construction yields a Riemannian metric in an open neighbourhood $\mathcal{V}_{\bar{m}}$ of $\bar m$. 
By the non-degeneracy assumption, we may assume that the sets $\{ \mathcal{V}_{\bar{m}}\}_{\bar m \in \sfN_Y}$ are pairwise disjoint. 
Let $\mathcal{U}_{\bar{m}}$ be an open neighbourhood of $\bar m$ satisfying 
    $\overline{\mathcal{U}_{\bar{m}}} 
        \subseteq \mathcal{V}_{\bar{m}}$ 
and let $f_{\bar m} : M \to [0,1]$ be a smooth function on $M$ satisfying 
    $f_{\bar m}|_{\mathcal{U}_{\bar{m}}} = 1$ and 
    $f_{\bar m}|_{M\setminus \mathcal{V}_{\bar{m}}} = 0$. 
Using an arbitrary metric $(g_*)_{\alpha \beta}$ on $M$ and the function 
    $\tilde f := 1 - \sum_{\bar m\in \sfN_Y} f_{\bar m}$, 
we define
\begin{align}
    g_{\alpha \beta} 
        :=
    \sum_{\bar m\in \sfN_Y} 
        f_{\bar m}(g_{\bar{m}})_{\alpha \beta} 
        + \tilde f \tilde g_{\alpha \beta},
\end{align}
which yields a $C^k$ metric 
    $g_{\alpha \beta}$ on $M$ 
satisfying 
    $g_{\alpha \beta}|_m = (g_{\bar{m}})_{\alpha \beta}|_m$ 
for all $\bar{m}\in \sfN_Y$ and $m\in \mathcal{U}_{\bar{m}}$. 

The crucial property of this background metric $g$,
which will be used below, 
is that 
the deficit 
    $R_\beta := Y_\beta 
        - g_{\alpha \beta} X^\alpha$ 
satisfies 
\begin{align}
    \label{eq:crucial}
    \partial_{c_1}\cdots \partial_{c_p}
        R_\beta|_{\bar{m}}
        =
    0
\end{align}
for all $\bar{m} \in \sfN_Y$ and $p\leq k+1$. 
This follows from the definition of the tensors 
$T^{a b}_{c_1 \cdots c_N}$
using the computation
\eqref{eq:derivatives-identity}.

\smallskip

\emph{Differentiability of the metric.} \
To verify that 
    $\tilde g_{\alpha \beta}$ is $k$ times continuously differentiable, 
we will show that the partial derivatives
\begin{align*}
    U_{\alpha \beta\, c_1\dots c_p}
        :=
    \partial_{c_1} \cdots \partial_{c_p}
            \frac{R_\alpha Y_\beta}
             {X^\delta Y_\delta} 
             \tand 
    V_{\alpha \beta\, c_1\dots c_p}
        := 
    \partial_{c_1} \cdots \partial_{c_p}
                \frac{R_\gamma X^\gamma 
                        Y_\alpha Y_\beta}
                    {(X^\delta Y_\delta)^2}
\end{align*}
can be continuously extended from $M \setminus \sfN_Y$ 
to all of $M$ for $p\leq k$. 
In view of \eqref{eq:g-formula} this yields the desired result.

We use the notation 
from Definition \ref{def:power-series}, thus
$\partial_{c_S} = \partial_{c_{i_1}} \cdots \partial_{c_{i_q}}$
for $S = \{ i_1, \ldots, i_q \} \subseteq \{1, \ldots, p\}$ with $i_\mu \neq i_\nu$ for $\mu \neq \nu$.  
With this notation we have
\begin{align*}
& U_{\alpha \beta\, c_1\dots c_p} 
    =  
    \sum_{\ell=0}^p 
    \sum_{\{S_1,\dots,S_\ell,A,B\}
            \in \mathcal{X}_p} 
        \frac{(-1)^\ell \ell!}
            {\left(X^\delta Y_\delta\right)^{\ell+1}} \partial_{c_{S_1}}
                  \! 
                \big(X^\delta Y_\delta\big)
                    \cdots \partial_{c_{S_\ell}}
                  \! 
            \big(X^\delta Y_\delta\big)
                \,
            \partial_{c_{A}} 
                R_\alpha
                \,
            \partial_{c_{B}}
                Y_\beta,    \\
&  V_{\alpha \beta\, c_1\dots c_p} 
        = 
    \sum_{\ell=0}^p
        \sum_{\{S_1,\dots,S_\ell,A,B\}
                \in \mathcal{X}_p}
             \!  
        \frac{(-1)^\ell \ell!}
            {\big(X^\delta Y_\delta\big)^{2(\ell+1)}}\partial_{c_{S_1}}
                   \! 
            \big(X^\delta Y_\delta\big)^2
            \cdots  
            \partial_{c_{S_\ell}}
                  \!
            \big(X^\delta Y_\delta\big)^2
                \partial_{c_{A}} 
                 \!  
                R_\gamma
            \,
            \partial_{c_{B}}
                 \! 
            \big(X^\gamma  Y_\alpha Y_\beta\big),
\end{align*}         
where $\mathcal{X}_p$ is the collection of all possible partitions of $\{1,\dots,p\}$. 

Let us fix a critical point 
    $\bar m \in \sfN_Y$ 
and let 
    $\bar{x}$ 
be the corresponding point in $\R^n$. 
Recall from \eqref{eq:XY-eq} that
\begin{align*}
    \big(X^\delta Y_\delta \big)^{-1}(x) 
        = 
        O \big(|x-\bar x|^{-2}\big).
\end{align*}
Furthermore, since
    $X^\alpha|_{\bar{x}} = 0$ 
and 
    $Y_\alpha|_{\bar{x}} = 0$,
Taylor's formula yields, for any $S \subseteq \{1, \ldots, p\}$,
\begin{align*}
    \partial_{c_S}
    \big(X^\delta Y_\delta\big)(x)
        =
    O \big( |x-\bar x|^{(2 - |S|)_+} \big),
    & & &
    \partial_{c_S}
        Y_\beta(x) 
    = O \big( |x-\bar x|^{(1 - |S|)_+} \big),\\
\partial_{c_S}
\big(X^\delta Y_\delta\big)^2(x)
= O \big(|x-\bar x|^{(4-|S|)_+}\big),
& & &
    \partial_{c_S}
            \big( X^\gamma Y_\alpha Y_\beta \big)(x)
        = O \big( |x-\bar x|^{(3 - |S|)_+} \big)
    .
\end{align*}
To estimate $\partial_{c_S} R_\alpha(x)$ we use the crucial point, observed in \eqref{eq:crucial}, that our background metric is constructed so that 
$\partial_{c_S}
        R_\beta(\bar x)
        =
    0$ 
when
$|S| \leq k+1$.
This ensures that
\begin{align*}
    \partial_{c_S}
        R_\alpha(x)
    = O \big(|x-\bar x|^{k+2-|S|}\big). 
\end{align*}
Combining these bounds, we estimate the right-hand sides of 
    $U_{\alpha \beta\, c_1\dots c_p}$ 
and 
    $V_{\alpha \beta\, c_1\dots c_p}$ 
as follows:
\begin{align*}
    \frac{1}
            {\left(X^\delta Y_\delta\right)^{\ell+1}} \partial_{c_{S_1}}
                  \! 
                \big(X^\delta Y_\delta\big)
                    \cdots \partial_{c_{S_\ell}}
                  \! 
            \big(X^\delta Y_\delta\big)
                \,
            \partial_{c_{A}} 
                R_\alpha
                \,
            \partial_{c_{B}}
                Y_\beta
    = O \big(|x-\bar x|^u\big),
    \\
    \frac{1} 
    {\left(X^\delta Y_\delta\right)^{2(\ell+1)}}
    \,
    \partial_{c_{S_1}} 
    \! 
    \big( X^\delta Y_\delta \big)^2 
        \cdots  
    \partial_{c_{S_\ell}}
    \!
    \big( X^\delta Y_\delta \big)^2
    \,
    \partial_{c_{A}} \! R_\gamma 
    \,
    \partial_{c_{B}}
    \! 
    \big( X^\gamma  Y_\alpha Y_\beta \big)
     =
    O\big( 
            |x-\bar x|^{
            v
            }
    \big),
\end{align*}
where the exponents $u$ and $v$ satisfy 
\begin{align*}
    u & = 
    -2 (\ell + 1)
    + \big(  2 - |S_1| \big)_+
    + \ldots
    + \big(  2 - |S_\ell| \big)_+
    + \big(  k + 2 - |S_A| \big)
    + \big(  1 - |S_B| \big)_+ \,, \\
    v  & = 
            - 4 (\ell  + 1)
        +
            \big(4-|S_1|\big)_+
            +
            \ldots
            +
            \big(4-|S_\ell|\big)_+
            +
            \big(  k + 2 - |S_A| \big)
            +
            \big( 3 - |S| \big)_+\,.
\end{align*}
Since 
    $|S_1|+ \cdots + |S_\ell| + |A| + |B| = p$ 
for $\{S_1, \ldots, S_\ell, A, B\} 
    \in \mathcal{X}_p$, 
we obtain 
$
    u \geq  k - p  + 1
    \geq 1
$
and 
$
    v \geq  k - p  + 1
    \geq 1
$,
which shows that 
\begin{align*}
    U_{\alpha \beta\, c_1\dots c_p}
    =
    O\big( |x - \bar x | \big)
    \tand
    V_{\alpha \beta\, c_1\dots c_p}
    =
    O\big( |x - \bar x | \big).
\end{align*}
Therefore 
    $U_{\alpha \beta\, c_1\dots c_p}$ 
and $V_{\alpha \beta\, c_1\dots c_p}$ can be extended continuously to all of $M$
by assigning the value zero for $\bar m \in \sfN_Y$.
\end{proof}

\section{Application to Quantum Markov Semigroups (QMS)}
\label{sec: QMS}

In this section prove Theorem \ref{thm:QMS-intro} by an application of Corollary \ref{Corollary: Main}.
As in Section \ref{sec:intro}, let $\cL$ be the generator of an ergodic quantum Markov semigroup $(\cP_t)_{t \geq 0}$ on a finite dimensional $C^*$-algebra $\cA$ with 
stationary state $\sigma \in \Dens_+$. 
The manifold under consideration is the set of strictly positive density matrices
\begin{align*}
    \Dens_+ = \{ \rho \in \Dens \ : \ 
        \rho > 0        
        \}.
\end{align*}
Note that $\Dens_+$ is a relatively open subset of the affine space $\sigma + T \subseteq \cA$, where
\begin{align*}
    T := 
    \{A\in \cA \ : \ 
        A = A^*, \
        \Tr[A] = 0
    \}.
\end{align*}
Therefore, the tangent space 
of $\Dens_+$ can be naturally identified with $T$. 
We will apply Corollary \ref{Corollary: Main} to the triple $(M, f, X)$ where $M := \Dens_+$ and
\begin{align*}
     f & : \Dens_+ \to \R, &&
        f(\sigma) := H_\sigma(\rho) 
        = \Tr[\rho (\log \rho - \log \sigma)],
\\
    X & : \Dens_+ \to T, &&
        X(\rho) := \cL^\dagger \rho.
\end{align*}

The functional $H_\sigma$ is everywhere strictly positive, except at its global minimum $\sigma$.
Moreover, a standard computation shows that, 
for $\rho \in \Dens_+$ and $A \in T$,
\begin{align}
    \label{eq:ent-diff}
    \partial_\eps\big|_{\eps = 0}
            H_\sigma(\rho + \eps A)
            = \Tr[(\log \rho - \log \sigma)A],
\end{align}
Therefore, the differential of $H_\sigma$ is everywhere non-zero except at $\sigma$, so that we are in a position to apply Corollary \ref{Corollary: Main}.

Recall that we are interested in the \textsc{bkm}-scalar product on $\cA$ given by 
\begin{align*}
    \ip{A,B}_\sigma^{\textsc{bkm}} :=
         \Tr[ A^* \cM_\sigma (B)],
        \quad \text{where }
    \cM_\sigma(B) := 
        \int_0^1 
            \sigma^{1-s} B\sigma^s 
        \dd s,
\end{align*}
for $A, B \in \cA$.
We refer to \cite{Amorim-Carlen:2021} for a recent study of this scalar product. 
It is natural to also consider the inner product on $\cA$ defined in terms of the inverse operator $\cM_\sigma^{-1} : \cA \to \cA$  given by
\begin{align*}
    \ip{A,B}_\sigma^{\widetilde{\textsc{bkm}}} :=
    \Tr[ A^* \cM_\sigma^{-1} (B)],
   \quad \text{where }
    \cM_\sigma^{-1}(B) 
    :=  
    \int_0^\infty
        (t + \sigma)^{-1} B  (t + \sigma)^{-1} 
    \dd t.
\end{align*}
We will use the following simple result.

\begin{lemma}
    \label{lem:SA-equiv}
For a linear operator $\cK : \cA \to \cA$ 
the following assertions are equivalent:
\begin{enumerate}
    \item $\cK$ is selfadjoint with respect to the inner product $\ip{\cdot,\cdot}_\sigma^{\textsc{bkm}}$.
    \item $\cK^\dagger$ is selfadjoint with respect to the inner product $\ip{\cdot,\cdot}_\sigma^{\widetilde{\textsc{bkm}}}$.
\end{enumerate}
\end{lemma}

\begin{proof}
It is readily seen that both assertions are equivalent to
$\cM_\sigma \cK = \cK^\dagger \cM_\sigma$.
\end{proof}

The \emph{entropy production functional}
    $I_\sigma : \Dens_+ \to \R$
is defined by 
\begin{align*}
    I_\sigma(\rho) = - \Tr[ (\log \rho - \log \sigma) \cL^\dagger \rho]
\end{align*}
for $\rho \in \Dens_+$.
Note that indeed 
$
    \ddt H_\sigma(\cP_t^\dagger \rho) 
        = - I_\sigma(\cP_t^\dagger \rho)
$.
The functional $I_\sigma$ is nonnegative and convex \cite{Spohn:1978,Spohn-Lebowitz:1978}. 
The following result shows the \emph{strict} positivity of the entropy production (except at stationarity) under the assumption of \textsc{bkm}-detailed balance.

\begin{proposition}
\label{prop:pos-entropy-production}
    Let $\cL$ be the generator of an ergodic quantum Markov semigroup on a finite dimensional $C^*$-algebra $\cA$, 
    with invariant state $\sigma \in \Dens_+$. 
    If \textsc{bkm}-detailed balance holds, 
    then $I_\sigma(\rho) > 0$ for all $\rho \in \Dens_+$ with $\rho \neq \sigma$.
\end{proposition}

\begin{proof}
    As remarked above, $I_\sigma$ is nonnegative and convex. Therefore, it suffices to show that $I_\sigma$ is strictly convex at its minimum $\sigma$.
    Take $A \in T$ with $A \neq 0$.

    For $\rho \in \Dens_+$  we set $\rho_\eps := \rho + \eps A$ for $|\eps|$ sufficiently small to ensure that $\rho_\eps \in \Dens_+$.
    Using the standard identities
    \begin{align*}
    \partial_\eps\big|_{\eps = 0} 
        \log \rho_\eps
            = \int_0^\infty  (t + \rho)^{-1} A  (t + \rho)^{-1} \dd t
            \tand
            \partial_\eps\big|_{\eps = 0} 
            (s + \rho_\eps)^{-1}
                = - (s + \rho)^{-1} A  (s + \rho)^{-1}
    \end{align*} 
    for $s > 0$, we obtain
    \begin{align*}
        \partial_\eps\big|_{\eps = 0}
            I_\sigma(\rho_\eps)
        = \Tr[  (\log \rho - \log \sigma) \cL^\dagger A ]
        + \Tr\bigg[  \int_0^\infty  (t + \rho)^{-1} A  (t + \rho)^{-1}  \cL^\dagger  \rho \dd t \bigg],
    \end{align*}
    and 
    \begin{align*}
        \partial_\eps^2\big|_{\eps = 0}
             I_\sigma(\rho_\eps)
         & =2 \Tr\bigg[  \int_0^\infty  (t + \rho)^{-1} A  (t + \rho)^{-1}  \cL^\dagger  A 
        \dd t \bigg]
        \\ & \qquad
        - 2  
        \Tr\bigg[  \int_0^\infty 
         (t + \rho)^{-1} A  (t + \rho)^{-1} A
         (t + \rho)^{-1}
          \cL^\dagger  \rho \dd t \bigg].
    \end{align*}
    In particular, for $\sigma_\eps := \sigma + \eps A$, we obtain
    \begin{align*}
        \partial_\eps^2\big|_{\eps = 0}
             I_\sigma(\sigma_\eps)
         & =2 \Tr\bigg[  \int_0^\infty  (t + \sigma)^{-1} A  (t + \sigma)^{-1}  \cL^\dagger  A 
        \dd t \bigg]
        = 
        2 \ip{A, \cL^\dagger A}_\sigma^{\widetilde{\textsc{bkm}}}.
    \end{align*}
    Since $I_\sigma$ is convex, this identity implies that $\ip{A, \cL^\dagger A}_\sigma^{\widetilde{\textsc{bkm}}} \geq 0$.
    
    On the other hand, $\cL^\dagger$
    is selfadjoint with respect to $\ip{\cdot,\cdot}_\sigma^{\widetilde{\textsc{bkm}}}$
    by Lemma \ref{lem:SA-equiv} and the assumption of
    \textsc{bkm}-detailed balance.
    Moreover, the restriction of $\cL^\dagger$ to $T$ is invertible by the ergodicity assumption.
    Therefore, 
        $\ip{A, \cL^\dagger A}_\sigma^{\widetilde{\textsc{bkm}}} \neq 0$.
    
    We thus conclude that 
    $\ip{A, \cL^\dagger A}_\sigma^{\widetilde{\textsc{bkm}}} > 0$, which yields the result.
\end{proof}

\begin{proof}[Proof of Theorem \ref{thm:QMS-intro}]
    First we will translate condition $(iii)$ of Corollary \ref{Corollary: Main}, namely the selfadjointness of the linearised operator $\Lambda$ with respect to the Hessian scalar product $h$. We claim that this is exactly the assumption of \textsc{bkm}-detailed balance in our setting. 
   
    Indeed, since $\cL^\dagger$ is a linear operator, its linearisation 
    $\Lambda : T \to T$
    appearing in condition $(iii)$
    is simply given by 
    $\Lambda := \cL^\dagger$. 
Moreover, the Hessian of $\rho \mapsto H_\sigma(\rho)$ at $\rho = \sigma$ is given by 
\begin{align*}
    h(A, B)
        :=
    \partial_\eps\big|_{\eps = 0} 
    \partial_\eta\big|_{\eta = 0} 
        H_\sigma\big( \sigma + \eps A + \eta B \big)
        = 
    \int_0^\infty 
        \Tr\Big[ 
            \frac{1}{s + \sigma}  
                A
            \frac{1}{s + \sigma}  
                B
        \Big]  \dd s
        = \ip{A,B}_\sigma^{\widetilde{\textsc{bkm}}}
\end{align*}
for $A, B \in T$. 
Hence the Hessian scalar product in condition $(iii)$ is the ${\widetilde{\textsc{bkm}}}$-scalar product. 
Thus, condition $(iii)$ is the ${\widetilde{\textsc{bkm}}}$-selfadjointness of $\cL^\dagger$. 
By Lemma \ref{lem:SA-equiv} this corresponds to the ${{\textsc{bkm}}}$-selfadjointness of $\cL$, which is the assumption of ${{\textsc{bkm}}}$-detailed balance.

This argument shows that the necessity of \textsc{bkm}-detailed balance for the gradient flow structure follows from Corollary \ref{Corollary: Main}. 
To show that \textsc{bkm}-detailed balance is also sufficient, we note first that condition $(ii)$ of Corollary \ref{Corollary: Main} is simply the stationarity condition $\cL^\dagger \sigma = 0$, which holds by assumption.
Thus, it remains to show that
condition $(i)$ of Corollary \ref{Corollary: Main} is implied by the assumption of ${{\textsc{bkm}}}$-detailed balance.
Then the existence of the gradient flow structure follows by applying Corollary \ref{Corollary: Main} in the opposite direction.

For this purpose, recall that $f = H_\sigma$ and $X = \cL^\dagger$, so that
\begin{align*}
    \nabla_{X} f
    = \Tr[(\log \rho - \log \sigma) \cL^\dagger \rho] 
    = - I_\sigma.   
\end{align*}
Hence, condition $(i)$ is the strict positivity of the entropy production $I_\sigma(\rho)$  or $\rho \neq \sigma$, which follows from the assumption of ${{\textsc{bkm}}}$-detailed balance by Proposition \ref{prop:pos-entropy-production}. 
\end{proof}

\subsection*{Acknowledgement} {\small 
J.M. gratefully acknowledges support by the European Research Council (ERC) under the European Union's Horizon 2020 research and innovation programme (grant agreement No 716117), and by the Austrian Science Fund (FWF), Project SFB F65.
}


\end{document}